\documentclass[12pt]{article}
\usepackage{amssymb,amsmath,amsfonts,amsthm,latexsym}

\begin{document}

\newtheorem{thm}{Theorem} %
\newtheorem{lem}[thm]{Lemma} %
\newtheorem{cor}[thm]{Corollary} %
\newtheorem{prop}[thm]{Proposition} %

\theoremstyle{remark} %
\newtheorem*{rem}{Remark} %

\def\C{{\mathbb C}} %
\def\F{{\mathbb F}} %
\def\N{{\mathbb N}} %
\def\Q{{\mathbb Q}} %
\def\R{{\mathbb R}} %
\def\Z{{\mathbb Z}} %
\def\e{{\mathbf e}} %
\def\x{\bold x} %

\def\eps{\varepsilon} %
\def\Tr{\operatorname{Tr}} %

\def\({\left(} %
\def\){\right)} %
\def\[{\left[} %
\def\]{\right]} %

\def\fl#1{\left\lfloor#1\right\rfloor} %
\def\rf#1{\left\lceil#1\right\rceil} %


\title{\bf Partial Gaussian Sums in Finite Fields}

\author{{\sc Ke Gong} \\
        Department of Mathematics, Henan University \\
        Kaifeng 475004, P. R. China \\
        {\tt kg@henu.edu.cn}}

\date{}

\maketitle

\begin{abstract}
We generalize Burgess' results on partial Gaussian sums to arbitrary finite fields. The main
ingredients are the classical method of amplification, two deep results on multiplicative energy
for subsets in finite fields which are obtained respectively by the tools from additive
combinatorics and geometry of numbers, and a technique of Chamizo for treating the difficulty
caused by additive character. Our results include the recent works on character sums in finite
fields by M.-C.~Chang and S.~V.~Konyagin.
\end{abstract}

\section{Introduction}

Let $p$ be a prime, $\chi$ a non-principal character modulo $p$. We denote $e_p(y):=\exp(2\pi
iy/p)$ as usual. Sums of the form
\begin{equation}\label{eq:definition}
\sum_{x=N}^{N+H}\chi(x)e_p(ax),
\end{equation}
are often encountered in analytic number theory.

We call the sums~\eqref{eq:definition} {\it pure} character sums if $a\equiv 0\pmod{p}$, otherwise
{\it mixed} character sums. If $H=p$ we say the sums~\eqref{eq:definition} {\it complete},
otherwise {\it incomplete} (or {\it partial} as Burgess used).

In the case of $a\not\equiv 0\pmod{p}$ and $H<p$, sums~\eqref{eq:definition} are usually called
partial Gaussian sums, which have been well studied by Vinogradov~\cite{V} and Burgess~\cite{B88}.
In this paper we try to generalize Burgess' results to arbitrary finite fields.

By a well-known generalization of the P\'olya-Vinogradov inequality we have
$$
\sum_{x=N}^{N+H}\chi(x)e_p(ax)\ll p^{1/2}\log p.
$$
For pure character sums, it was shown by Burgess \cite{B63} that for any positive integer $r$ we
have
\begin{equation}\label{eq:Burgess_a=0}
\sum_{x=N}^{N+H}\chi(x)\ll H^{1-1/r}p^{(r+1)/{4r^2}}\log p.
\end{equation}
Fifteen years later, by a modification of his method in proving~\eqref{eq:Burgess_a=0},
Burgess~\cite{B88} proved the following estimates for general partial Gaussian sums.

\begin{thm}\label{thm:Burgess}
Let $\chi$ be a non-principal character modulo a prime $p$. Then for any integers $r\ge 2$, $a$,
$N$ and $1\le H < p$ we have
\begin{equation}\label{eq:Burgessgeneral}
\sum_{x=N}^{N+H}\chi(x)e_p(ax)\ll H^{1-1/r}p^{1/{4(r-1)}}\log^2 p.
\end{equation}
\end{thm}

On the other hand, parallel to the pure character sums~\eqref{eq:Burgess_a=0} in prime field
$\F_p$, there are also many works on pure character sums in general finite fields $\F_q$, $q=p^n$.
See the papers of Davenport and Lewis~\cite{DL}, Chang~\cite{Ch1} and Konyagin~\cite{Kon}. So it
is naturally to consider partial Gaussian sums in arbitrary finite fields. However, such a
generalization is quite unusually because the additive character $e_p(\cdot)$ causes additional
difficulty even in the case of prime field. Indeed Burgess himself has remarked that {\it the
argument used to obtain~\eqref{eq:Burgess_a=0} depended on the summand being multiplicative} (see
Burgess~\cite[p. 589]{B88}). Thus the method used by Burgess does not have any natural extensions
to the case of arbitrary finite fields. And even nowadays, although the results we obtain in this
paper match Burgess' results in the same range, they are not as explicit as those of Burgess.

Recently, Chamizo~\cite{C} presented a new proof of Burgess' partial Gaussian sums on the Third
Conference on Number Theory at University of Salamanca (Salamanca, July 2009). Chamizo's used
essentially the classical {\it method of amplification}\footnote{The method of amplification was
first used in number theory by Vinogradov~\cite{Vino}, then introduced by Karatsuba~\cite{K} into
the study of character sums. Now it is a classical method, see Friedlander~\cite{F}, Iwaniec and
Kowalski~\cite{IK}, Chang~\cite{Ch1}.} in the form of Iwaniec and Kowalski~\cite{IK}. He
ingeniously introduced a trick to overcome the difficulty caused by additive character.

In the present paper we generalize Burgess' partial Gaussian sums to arbitrary finite fields. Two
deep results on multiplicative energy for subsets in finite fields, which are obtained
respectively by some tools from additive combinatorics and geometry of numbers, are involved here.
We will also use Chamizo's trick.

We finally remark that Perel'muter~\cite{P} has studied partial Gaussian sums over additive
subgroup of $\F_{p^n}$. However he mainly concerned with the algebraic respects.

\section{Notation}

Throughout the paper we will use the following notations.

Let $p$ be an odd prime, $q$ an integer with $q=p^n$, and $\F_p$ the prime field. Let $\F_q$
denote the finite field with $q$ elements.

We recall that the function
$$
\Tr(z)=\sum_{i=0}^{n-1}z^{p^i}
$$
is called the {\it trace} of $z\in\F_{p^n}$ over $\F_p$.

Define $e_p(z)=\exp(2\pi iz/p)$. Then the set of functions $\psi_a(z)=e_p(\Tr(az))$,
$a\in\F_{p^n}$, form the set of additive characters of $\F_{p^n}$, with $\psi_0$ being the trivial
character.

Let $\chi$ be a nontrivial multiplicative character of $\F_{p^n}$.

Let $\{\omega_1,\ldots,\omega_n\}$ be an arbitrary basis for $\F_{p^n}$ over $\F_p$. Then the
elements of $\F_{p^n}$ have a unique representation as
\begin{equation}\label{eq:elements}
\xi=x_1\omega_1+\cdots+x_n\omega_n,\qquad 0\le x_i<p.
\end{equation}
We denote by $B$ a box in the $n$-dimensional space, defined by
\begin{equation}\label{eq:box}
N_j<x_j\le N_j+H_j,\qquad 1\le j\le n,
\end{equation}
where $N_j,H_j$ are integers satisfying $0\le N_j<N_j+H_j<p$ for all $j$.

For $A\subset\F_q$, we denote by
\begin{equation}\label{eq:multiplicetive_energy}
E(A):=|\{(x_1,x_2,x_3,x_4)\in A\times A\times A\times A: x_1x_2=x_3x_4\}|
\end{equation}
the {\it multiplicative energy} of $A$.

As usual, `$O$' and `$\ll$' denote respectively Landau and Vinogradov symbol, in which the
constants implied depend only on $n$ throughout this paper.

\section{Preliminary}

\subsection{Pure character sums in finite fields}

Davenport and Lewis~\cite{DL} proved in 1963 that

\begin{thm}\label{thm:Davenport&Lewis}
Let $H_j=H$ for $1\le j\le n$ with
$$
H>p^{\frac{n}{2(n+1)}+\eps}\quad\text{for some}~\eps>0,
$$
and let $p>p_1(\eps)$, then
$$
\left|\sum_{x\in B}\chi(x)\right|<(Hp^{-\delta})^n,
$$
where $\delta=\delta(\eps)>0$.
\end{thm}

\begin{rem} We see that if $n=1$, the exponent in Theorem~\ref{thm:Davenport&Lewis} is still $1/4+\eps$,
which recovers Burgess' result. While as $n$ increases, the exponent $\frac{n}{2(n+1)}$ will be
near to $1/2$.
\end{rem}

About two years ago, M.-C.~Chang wrote a series of papers to introduce some tools from additive
combinatorics, mainly the sum-product theorems in finite fields, into the study of character sums
estimates. She obtained many interesting results, one of which improved Davenport and
Lewis~\cite{DL} by combining Burgess' classical amplification method with some estimates for
multiplicative energy for subsets in $\F_{p^n}$.

\begin{thm}\label{thm:Chang}
Let $\chi$ be a nontrivial multiplicative character of $\F_{p^n}$. Given $\eps>0$, there is
$\tau>\eps^2/4$ such that if
$$
B=\left\{\sum_{j=1}^nx_j\omega_j:x_j\in (N_j,N_j+H_j]\cap\Z,\, 1\le j\le n\right\}
$$
is a box satisfying
$$
\prod_{j=1}^n H_j>p^{(\frac25+\eps)n}
$$
then for $p>p(\eps)$,
$$
\left|\sum_{x\in B}\chi(x)\right|\ll_n |B|p^{-\tau},
$$
unless $n$ is even and $\chi\mid_{F_2}$ is principal, where $F_2$ is the subfield of size
$p^{n/2}$, in which case,
$$
\left|\sum_{x\in B}\chi(x)\right|\le\max_{\xi}|B\cap\xi F_2|+O_n(|B|p^{-\tau}).
$$
\end{thm}

\begin{rem}
Theorem~\ref{thm:Chang} also holds if we replace the assumption $\prod_{j=1}^n H_j>p^{(\frac25+\eps)n}$
by the stronger one
$$
H_j > p^{2/5+\eps},\qquad\text{for all}~j,
$$
which improved upon Davenport and Lewis~\cite{DL} for $n>4$. But, for higher-dimensional
generalization, the results do not achieve the strength of Burgess~\cite{B62}.

We note that Burgess' strength is obtained only for some special cases, see Burgess~\cite{B67},
Karatsuba~\cite{K} and Chang~\cite{Ch2}.
\end{rem}

The main ingredient in Chang~\cite{Ch1} is the following estimate for the multiplicative energy.

\begin{prop}\label{prop:ChangEnergy} Let $\{\omega_1,\ldots,\omega_n\}$ be a basis for
$\F_{p^n}$ over $\F_p$, and let $B\subset\F_{p^n}$ be the box
$$
B=\left\{\sum_{j=1}^nx_j\omega_j:x_j\in[N_j+1,N_j+H_j],\, j=1,\ldots,n\right\}
$$
where $1\le N_j<N_j+H_j<p$ for all $j$. Assume that
\begin{equation}\label{eq:condition}
\max_j H_j<\frac12(\sqrt{p}-1).
\end{equation}
Then we have
$$
E(B,B)<C^n(\log p)|B|^{11/4}
$$
for an absolute constant $C<2^{9/4}$.
\end{prop}

\begin{rem}
Using a result of Perel'muter and Shparlinski~\cite{PS} and some sophisticated arguments, Chang
removed the influence of the condition~\eqref{eq:condition} on Theorem~\ref{thm:Chang}.
\end{rem}

On the conference of 26th Journ\'ees Arithm\'etiques (Saint-Etienne, July 2009), using the method
in geometry of numbers (see~\cite{Ban}, \cite{TV}), Konyagin~\cite{Kon} improved Chang's estimate
for multiplicative energy if $H_i=H$, $1\le i\le n$.

\begin{prop}\label{prop:KonyaginEnergy}
If $H_1=\cdots=H_n\le p^{1/2}$, then
$$
E(B)\ll |B|^2\log p.
$$
\end{prop}

Then, incorporating the estimate of Proposition~\ref{prop:KonyaginEnergy} into Burgess'
amplification process, Konyagin proved

\begin{thm}\label{thm:Konyagin}
Let $\chi$ be a nontrivial multiplicative character of $\F_{p^n}$ and $0<\eps\le 1/4$ be given. If
$n\ge 2$ and $B$ is a box defined in~\eqref{eq:box} and satisfying
$$
H_j\ge p^{1/4+\eps},\qquad 1\le j\le n,
$$
then
$$
\left|\sum_{x\in B}\chi(x)\right|\ll |B|p^{-\eps^2/2}.
$$
\end{thm}

\subsection{Weil's theorem}

We will need the following version of Weil's bound on exponential sums. See~\cite[Theorem
11.23]{IK}.

\begin{thm}[A. Weil]\label{thm:Weil} Let $\chi$ be a nontrivial multiplicative character of $\F_{p^n}$ of order $d>1$.
Suppose $f\in\F_{p^n}[x]$ has $m$ distinct roots and $f$ is not a $d$-th power. Then for $n\ge 1$
we have
$$
\left|\sum_{x\in\F_{p^n}}\chi(f(x))\right| \le (m-1)p^{\frac{n}{2}}.
$$
\end{thm}

\section{Main results}

The following two theorems generalize Theorem~\ref{thm:Chang} and Theorem~\ref{thm:Konyagin}
respectively.

\begin{thm}\label{thm:CG}
Let $\chi$ be a nontrivial multiplicative character of $\F_{p^n}$. Given $\eps>0$, there is
$\tau>\eps^2/4$ such that if $B$ is a box defined in~\eqref{eq:box} and satisfying
$$
\prod_{j=1}^n H_j \ge p^{\(\frac25+\eps\)n},
$$
then for $p>p(\eps)$,
$$
\left|\sum_{x\in B}\chi(x)e_p(\Tr(ax))\right|\ll |B|p^{-\tau},
$$
unless $n$ is even and $\chi|_{F_2}$ is principal, where $F_2$ is the subfield of size $p^{n/2}$,
in which case,
$$
\left|\sum_{x\in B}\chi(x)e_p(\Tr(ax))\right|\le\max_{\xi}|B\cap\xi F_2| + O_n(p^{-\tau}|B|).
$$
\end{thm}

\begin{thm}\label{thm:KG}
Let $\chi$ be a nontrivial multiplicative character of $\F_{p^n}$ and $0<\eps\le 1/4$ be given. If
$n\ge 2$ and $B$ is a box defined in~\eqref{eq:box} and satisfying
$$
H_j\ge p^{1/4+\eps},\qquad 1\le j\le n,
$$
then
$$
\left|\sum_{x\in B}\chi(x)e_p(\Tr(ax))\right|\ll |B|p^{-\eps^2/2}.
$$
\end{thm}

\section{Proof of Theorem~\ref{thm:CG}}

We incorporate the technique of Chamizo~\cite{C} into the argument
of Chang~\cite{Ch1}.

\begin{proof}[Proof of Theorem~\ref{thm:CG}] We first prove the theorem under the restriction
\begin{equation}\label{eq:H_j}
H_j<\frac12(\sqrt{p}-1)\qquad\text{for all}~j,
\end{equation}
which is inherited from the estimate for multiplicative energy in Proposition~\ref{prop:ChangEnergy}.

By breaking up $B$ in smaller boxes, we may assume
\begin{equation}\label{eq:Chang_H}
\prod_{j=1}^n H_j\sim p^{(\frac25+\eps)n}.
\end{equation}

Let $\delta>0$ be specified later. Let
$$
I=[1,p^{\delta}],
$$
$$
B_0=\left\{\sum_{j=1}^n x_j\omega_j : x_j\in[0,p^{-2\delta}H_j],\,j=1,\ldots,n \right\}
$$
and
$$
B_I=\left\{\sum_{j=1}^n x_j\omega_j : x_j\in[0, p/|I|^{\frac{1}{n}}],\, j=1,\ldots,n \right\}.
$$

Since $B_0 I\subset\left\{\sum_{j=1}^n x_j\omega_j : x_j\in[0, p^{-\delta}H_j],\, j=1,\ldots,n
\right\}$, clearly
\begin{eqnarray*}
\lefteqn{\left|\sum_{x\in B}\chi(x)e_p(\Tr(ax)) - \sum_{x\in B}\chi(x+yz)e_p\(\Tr(a(x+yz))\)\right|} \\
& & \quad < |B\setminus(B+yz)| + |(B+yz)\setminus B| < 2np^{-\delta}|B|
\end{eqnarray*}
for $y\in B_0$, $z\in I$. Hence
\begin{eqnarray*}
\lefteqn{\sum_{x\in B}\chi(x)e_p(\Tr(ax))} \\
& & \qquad =\frac{1}{|B_0|\,|I|}\sum_{x\in B,\,y\in B_0,\,z\in
I}\chi(x+yz)e_p\(\Tr(a(x+yz))\)+O(np^{-\delta}|B|).
\end{eqnarray*}
We now estimate
\begin{eqnarray*}
& & \left|\sum_{x\in B,\,y\in B_0,\,z\in I}\chi(x+yz)e_p\(\Tr(a(x+yz))\)\right| \\
& \le & \sum_{x\in B,\,y\in B_0}\bigg|\sum_{z\in I}\chi(x+yz)e_p\(\Tr(a(x+yz))\)\bigg| \\
& \le & \sum_{u\in\F_q}\nu(u)\sup_{y\in B_0}\left|\sum_{z\in I}\chi(u+z)e_p\(\Tr(ay(u+z))\)\right| \\
& \le & \sum_{u\in\F_q}\nu(u)\sup_b\left|\sum_{z\in I}\chi(u+z)e_p\(\Tr(b(u+z))\)\right|,
\end{eqnarray*}
where
$$
\nu(u)=\Big|\Big\{(x,y)\in B\times B_0:\frac{x}{y}=u\Big\}\Big|.
$$

Then
\begin{eqnarray*}
\lefteqn{\sum_{x\in B}\chi(x)e_p(\Tr(ax))} \\
& & \qquad \le\frac{1}{|B_0|\,|I|}\sum_{u\in\F_q}\nu(u)\sup_b\left|\sum_{z\in I}\chi(u+z)e_p\(\Tr(b(u+z))\)\right| + O(np^{-\delta}|B|) \\
& & \qquad \le\frac{1}{|B_0|\,|I|}\sum_{u\in\F_q}\nu(u)\sup_b\frac{|I|}{q}\sum_{\substack{c\\ c-b\in B_I}}\left|\sum_{z\in I}\chi(u+z)e_p\(\Tr(c(u+z))\)\right| + O(np^{-\delta}|B|) \\ %
& & \qquad \le\frac{1}{|B_0|\,q}\sum_{u\in\F_q}\nu(u)\sum_{\substack{c\\ c-b_0\in B_I}}\left|\sum_{z\in I}\chi(u+z)e_p\(\Tr(c(u+z))\)\right| + O(np^{-\delta}|B|) %
\end{eqnarray*}
with the sum
$$
\sum_{\substack{c\\ c-b\in B_I}}\left|\sum_{z\in I}\chi(u+z)e_p\(\Tr(c(u+z))\)\right|
$$
attains its maximum at $b_0\in\F_q$.

Let $r\ge 2$ by any integer. Applying the H\"older inequality
$$
\sum_{u\in\F_q}\nu(u)\sum_{\substack{c\\ c-b_0\in B_I}}\left|\sum_{z\in I}\chi(u+z)e_p\(\Tr(c(u+z))\)\right|\le V_1^{1-\frac{1}{r}} V_2^{\frac{1}{2r}} W^{\frac{1}{2r}}, %
$$
where
$$
V_1=\sum_{u\in\F_q}\nu(u),\quad V_2=\sum_{u\in\F_q}\nu^2(u),
$$
$$
W=\sum_{u\in\F_q}\(\sum_{\substack{c\\ c-b_0\in B_I}}\left|\sum_{z\in I}\chi(u+z)e_p\(\Tr(c(u+z))\)\right|\)^{2r}. %
$$

Observe that
$$
V_1 = |B||B_0|
$$
and
\begin{eqnarray*}
V_2 %
& = & |\{(x_1,x_2,y_1,y_2)\in B\times B\times B_0\times B_0:x_1y_2=x_2y_1\}| \\
& = & \sum_v|\{(x_1,x_2):\frac{x_1}{x_2}=v\}|\,|\{(y_1,y_2):\frac{y_1}{y_2}=v\}| \\
& \le & E(B,B)^{\frac12}E(B_0,B_0)^{\frac12} \\
& < & 2^{\frac94n+1}(\log p)|B|^{\frac{11}{8}}|B_0|^{\frac{11}{8}} \\
& < & 2^{\frac94n+1}(\log p)|B|^{\frac{11}{4}}p^{-\frac{11}{4}n\delta},
\end{eqnarray*}
by the Cauchy-Schwarz inequality, Proposition~\ref{prop:ChangEnergy} and the definition of $B_0$.

Now we bound $W$. Recall that
$$
q=p^n.
$$
Then
\begin{eqnarray*}
\lefteqn{\sum_{u\in\F_q}\(\sum_{\substack{c\\ c-b_0\in B_I}}\left|\sum_{z\in I}\chi(u+z)e_p\(\Tr(c(u+z))\)\right|\)^{2r}} \\
& & \le (q/|I|)^{2r-1}\sum_{u\in\F_q}\sum_{c\in\F_q}\left|\sum_{z\in I}\chi(u+z)e_p\(\Tr(c(u+z))\)\right|^{2r} \\
& & \le (q/|I|)^{2r-1}\sum_{z_1,\,\ldots,\,z_{2r}\in I}\bigg|\sum_{u\in\F_q}\chi((u+z_1)\cdots(u+z_r)(u+z_{r+1})^{q-2}\cdots(u+z_{2r})^{q-2}) \\
& & \qquad \times\sum_{c\in\F_q}e_p\(\Tr(c(z_1+\cdots+z_r-z_{r+1}-\cdots-z_{2r}))\)\bigg| \\
& & = q^{2r}|I|^{1-2r}\sum_{\substack{z_1,\,\ldots,\,z_{2r}\in I\\ z_1+\cdots+z_r=z_{r+1}+\cdots+z_{2r}}} %
\bigg|\sum_{u\in\F_q}\chi((u+z_1)\cdots(u+z_r)(u+z_{r+1})^{q-2}\cdots(u+z_{2r})^{q-2})\bigg|.
\end{eqnarray*}
The last equality is by the orthogonality of additive characters.

For $z_1,\ldots,z_{2r}\in I$ such that at least one of the elements is not repeated twice, the
polynomial $f_{z_1,\ldots,z_{2r}}(u)=(u+z_1)\cdots(u+z_r)(u+z_{r+1})^{q-2}\cdots(u+z_{2r})^{q-2}$
clearly cannot be a $d$-th power. Since $f_{z_1,\ldots,z_{2r}}(u)$ has no more than $2r$ many
distinct roots, Theorem~\ref{thm:Weil} gives
$$
\bigg|\sum_{u\in\F_q}\chi((u+z_1)\cdots(u+z_r)(u+z_{r+1})^{q-2}\cdots(u+z_{2r})^{q-2})\bigg|<2rp^{\frac{n}{2}}.
$$

For those $z_1,\ldots,z_{2r}\in I$ such that every root of $f_{z_1,\ldots,z_{2r}}(u)$ appears at
least twice, we bound $\sum\bigl|\sum_{u\in\F_q}\chi(f_{z_1,\ldots,z_{2r}}(u))\bigr|$ by $q$ times
the number of such $z_1,\ldots,z_{2r}$. Since there are at most $r$ roots in $I$ and for each
$z_1,\ldots,z_{2r}$ there are at most $r$ choices, we obtain a bound $|I|^r r^{2r}p^n$.

Therefore
$$
\sum_{u\in\F_q}\(\sum_{c\in\F_q}\left|\sum_{z\in I}\chi(u+z)e_p\(\Tr(c(u+z))\)\right|\)^{2r}\le
q^{2r}|I|(r^{2r}|I|^{-r} p^n+2rp^{\frac{n}{2}})
$$
and
$$
W^{\frac{1}{2r}} %
\le q|I|^{\frac{1}{2r}}(r|I|^{-\frac12}p^{\frac{n}{2r}}+2p^{\frac{n}{4r}}).
$$

Putting the above estimates together, we have
\begin{eqnarray*}
\frac{1}{|B_0|q}V %
& < & 4^{\frac{n}{r}}(\log p)(|B_0||B|)^{-\frac{1}{r}}|B|^{1+\frac{11}{8r}}p^{-\frac{11}{8}\frac{n}{r}\delta}
|I|^{\frac{1}{2r}}\Big(r|I|^{-\frac12}p^{\frac{n}{2r}}+2p^{\frac{n}{4r}}\Big) \\
& < & 4^{\frac{n}{r}}(\log p)p^{2\frac{n}{r}\delta-\frac{11}{8}\frac{n}{r}\delta}|B|^{1-\frac{5}{8r}}
|I|^{\frac{1}{2r}}\Big(r|I|^{-\frac12}p^{\frac{n}{2r}}+2p^{\frac{n}{4r}}\Big) \\
& < & 2r\cdot4^{\frac{n}{r}}|B|^{1-\frac{5}{8r}}p^{\frac{n}{4r}+\frac58\frac{n}{r}\delta+\frac{\delta}{2r}}\log p \\
& < & 2r\cdot4^{\frac{n}{r}}|B|p^{\frac{n}{4r}+\frac58\frac{n}{r}\delta-\frac58\frac{n}{r}(\frac25+\eps)+\frac{\delta}{2r}}\log p \\
& < & 2r\cdot4^{\frac{n}{r}}|B|p^{-\frac58\frac{n}{r}(\eps-\delta)+\frac{\delta}{2r}}\log p.
\end{eqnarray*}
The second-to-last inequality holds because of \eqref{eq:Chang_H} and by assuming $\delta\ge\frac{n}{2r}$.

Similar to the argument of Chang~\cite{Ch1}, we can show that
$$
p^{-\frac58\frac{n}{r}(\eps-\delta)+\frac{\delta}{2r}}\log p < p^{-\eps^2/4}.
$$
Then we prove the theorem under the condition~\eqref{eq:H_j}.

\bigskip

Now we are at the position to remove the additional hypothesis~\eqref{eq:H_j} on the shape of $B$.
We proceed in several steps and rely essentially on a further key ingredient provided by the following
estimate in Perel'muter and Shparlinski~\cite{PS}.

\begin{prop}\label{prop:PS}
Let $\chi$ be a nonprincipal multiplicative character of $\F_q$ and let $g\in\F_q$ be a generating
element, i.e. $\F_q=\F_p(g)$. Then for any $a\in\F_p$, we have
\begin{equation}\label{eq:PS}
\left|\sum_{t\in\F_p}\chi(g+t)e_p(at)\right|\le np^{1/2}.
\end{equation}
\end{prop}

First we make the following observation.

Let $H_1\ge H_2\ge\ldots\ge H_n$. If $H_1<p^{\frac12+\frac{\varepsilon}{2}}$, we may clearly write
$B$ as a disjoint union of boxes $B_{\alpha}\subset B$ satisfying the first condition
in~\eqref{eq:H_j} and $|B_{\alpha}|>(\frac12
p^{-\frac{\varepsilon}{2}})^n|B|>2^{-n}p^{(\frac25+\frac{\varepsilon}{2})n}$. Since~\eqref{eq:H_j}
holds for each $B_{\alpha}$, we have
$$
\left|\sum_{x\in B_{\alpha}}\chi(x)e_p(\Tr(ax))\right|<cnp^{-\tau}|B_{\alpha}|.
$$
Hence
$$
\left|\sum_{x\in B}\chi(x)e_p(\Tr(ax))\right|<cnp^{-\tau}|B|.
$$
Therefore we may assume that $H_1>p^{\frac12+\frac{\varepsilon}{2}}$.

\bigskip
\noindent{\it Case 1. $n$ is odd.}
\bigskip

We denote $I_i=[N_i+1,N_i+H_i]$. Then
\begin{eqnarray*}
\lefteqn{\left|\sum_{x\in B}\chi(x)e_p(\Tr(ax))\right|} \\
& & = \left|\sum_{\substack{x_i\in I_i\\ 2\le i\le n}}\sum_{x_1\in I_1}\chi\(x_1+x_2\frac{\omega_2}{\omega_1}+\cdots+x_n\frac{\omega_n}{\omega_1}\) %
e_p(\Tr(ax_1\omega_1+x_2\omega_2+\cdots+x_n\omega_n))\right| \\
& & \le \left|\sum_{(x_2,\ldots,x_n)\in D^c}\sum_{x_1\in I_1}\chi\(x_1+x_2\frac{\omega_2}{\omega_1}+\cdots+x_n\frac{\omega_n}{\omega_1}\) %
e_p(\Tr(ax_1\omega_1+x_2\omega_2+\cdots+x_n\omega_n))\right| \\
& & \qquad + \left|\sum_{(x_2,\ldots,x_n)\in D}\sum_{x_1\in I_1}\chi\(x_1+x_2\frac{\omega_2}{\omega_1}+\cdots+x_n\frac{\omega_n}{\omega_1}\) %
e_p(\Tr(ax_1\omega_1+x_2\omega_2+\cdots+x_n\omega_n))\right|
\end{eqnarray*}
with
$$
D = \left\{(x_2,\ldots,x_n)\in I_2\times\cdots\times I_n : %
\F_p\(x_2\frac{\omega_2}{\omega_1}+\cdots+x_n\frac{\omega_n}{\omega_1}\)\neq\F_q\right\}
$$
and
$$
D^c = I_2\times\cdots\times I_n\setminus D.
$$

Using~\eqref{eq:PS} we estimate the first sum as
\begin{eqnarray*}
\lefteqn{\left|\sum_{(x_2,\ldots,x_n)\in D^c}\sum_{x_1\in I_1}\chi\(x_1+x_2\frac{\omega_2}{\omega_1}+\cdots+x_n\frac{\omega_n}{\omega_1}\) %
e_p(\Tr(ax_1\omega_1+x_2\omega_2+\cdots+x_n\omega_n))\right|} \\
& & = \left|\sum_{\substack{x_i\in I_i\\ 2\le i\le n}}e_p(\Tr(a(x_2\omega_2+\cdots+x_n\omega_n))) %
\sum_{x_1\in I_1}\chi\(x_1+x_2\frac{\omega_2}{\omega_1}+\cdots+x_n\frac{\omega_n}{\omega_1}\)e_p(\Tr(ax_1\omega_1))\right| \\
& & \le \sum_{\substack{x_i\in I_i\\ 2\le i\le n}} %
\left|\sum_{x_1\in I_1}\chi\(x_1+x_2\frac{\omega_2}{\omega_1}+\cdots+x_n\frac{\omega_n}{\omega_1}\)e_p(\Tr(a\omega_1)x_1)\right| \\
& & = \sum_{\substack{x_i\in I_i\\ 2\le i\le n}} %
\left|\sum_{x_1\in\F_p}\chi\(x_1+x_2\frac{\omega_2}{\omega_1}+\cdots+x_n\frac{\omega_n}{\omega_1}\)e_p(\Tr(a\omega_1)x_1) %
\cdot\frac{1}{p}\sum_{b\in\F_p}\sum_{x'_1\in I_1}e_p(b(x_1-x'_1))\right| \\
& & \le \frac{1}{p}\sum_{\substack{x_i\in I_i\\ 2\le i\le n}}\sum_{b\in\F_p} %
\left|\sum_{x_1\in\F_p}\chi\(x_1+x_2\frac{\omega_2}{\omega_1}+\cdots+x_n\frac{\omega_n}{\omega_1}\)e_p((\Tr(a\omega_1)+b)x_1)\right| %
\cdot\left|\sum_{x'_1\in I_1}e_p(-bx'_1)\right| \\
& & \le \frac{1}{p}np^{1/2}\frac{|B|}{H_1}\sum_{b\in\F_p}\left|\sum_{x'_1\in I_1}e_p(bx'_1)\right| \\
& & \le c(n)p^{\frac12}\log p\frac{|B|}{H_1}.
\end{eqnarray*}

For the second sum, we have
\begin{eqnarray*}
& & \left|\sum_{(x_2,\ldots,x_n)\in D}\sum_{x_1\in I_1}\chi\(x_1+x_2\frac{\omega_2}{\omega_1}+\cdots+x_n\frac{\omega_n}{\omega_1}\) %
e_p(\Tr(ax_1\omega_1+x_2\omega_2+\cdots+x_n\omega_n))\right| \\
& & \qquad \le p|D|\le p\sum_G\left|G\bigcap\operatorname{Span}_{\F_p}\(\frac{\omega_2}{\omega_1},\ldots,\frac{\omega_n}{\omega_1}\)\right|,
\end{eqnarray*}
where $G$ runs over nontrivial subfields of $\F_q$. Since $q=p^n$ and $n$ is odd, obviously
$[\F_q:G]\ge 3$. Hence $[G:\F_p]\le\frac{n}{3}$. Furthermore, since
$\{\omega_1,\ldots,\omega_n\}$ is a basis of $\F_q$ over $\F_p$,
$1\notin\operatorname{Span}_{\F_p}(\frac{\omega_2}{\omega_1},\ldots,\frac{\omega_n}{\omega_1})$
and the proceeding implies that
$$
\operatorname{dim}_{\F_p}\(G\bigcap\operatorname{Span}_{\F_p}
\(\frac{\omega_2}{\omega_1},\ldots,\frac{\omega_n}{\omega_1}\)\)\le\frac{n}{3}-1.
$$

Therefore, under our assumption on $|H_1|$, we have
$$
\left|\sum_{x\in B}\chi(x)e_p(\Tr(ax))\right|\le c(n)((\log p)p^{-\frac{\varepsilon}{2}}|B|+p^{\frac{n}{3}}) %
< \left(c(n)(\log p)p^{-\frac{\varepsilon}{2}}+p^{-\frac{n}{15}}\right)|B|,
$$
since $|B|>p^{\frac25 n}$. This proves our claim.

\bigskip
\noindent{\it Case 2. $n$ is even.}
\bigskip

In view of the earlier discussion, we have
\begin{eqnarray*}
\lefteqn{\left|\sum_{x\in B}\chi(x)e_p(\Tr(ax))\right|} \\
& & \le \left|\sum_{(x_2,\ldots,x_n)\in D_2^c}\sum_{x_1\in I_1}\chi\(x_1+x_2\frac{\omega_2}{\omega_1}+\cdots+x_n\frac{\omega_n}{\omega_1}\) %
e_p(\Tr(ax_1\omega_1+x_2\omega_2+\cdots+x_n\omega_n))\right| \\
& & \quad + \left|\sum_{(x_2,\ldots,x_n)\in D_2}\sum_{x_1\in I_1}\chi\(x_1+x_2\frac{\omega_2}{\omega_1}+\cdots+x_n\frac{\omega_n}{\omega_1}\) %
e_p(\Tr(ax_1\omega_1+x_2\omega_2+\cdots+x_n\omega_n))\right|,
\end{eqnarray*}
where
$$
D_2 = \left\{(x_2,\ldots,x_n)\in I_2\times\cdots\times I_n :
\(x_2\frac{\omega_2}{\omega_1}+\cdots+x_n\frac{\omega_n}{\omega_1}\)\in F_2\right\},
$$
$$
D_2^c = I_2\times\cdots\times I_n\setminus D_2,
$$
and $F_2$ is the subfield of size $p^{n/2}$.

Our only concern is to bound the second sum, namely
$$
\varpi=\left|\sum_{(x_2,\ldots,x_n)\in D_2}\sum_{x_1\in I_1}\chi\(x_1+x_2\frac{\omega_2}{\omega_1}+\cdots+x_n\frac{\omega_n}{\omega_1}\) %
e_p(\Tr(ax_1\omega_1+x_2\omega_2+\cdots+x_n\omega_n))\right|.
$$

First, we note that since $1,\frac{\omega_2}{\omega_1},\ldots,\frac{\omega_n}{\omega_1}$ are
independent, $\frac{\omega_j}{\omega_1}\in F_2$ for at most $\frac{n}{2}-1$ many $j$'s. After
reordering, we may assume that $\frac{\omega_j}{\omega_1}\in F_2$ for $2\le j\le k$
and $\frac{\omega_j}{\omega_1}\notin F_2$ for $k+1\le j\le n$, where
$k\le\frac{n}{2}$. we also assume that $H_{k+1}\le\cdots\le H_n$. Fix
$x_2,\ldots,x_{n-1}$. Obviously there is no more than one value of $x_n$ such that
$x_2\frac{\omega_2}{\omega_1}+\cdots+x_n\frac{\omega_n}{\omega_1}\in F_2$, since otherwise
$(x_n-x'_n)\frac{\omega_n}{\omega_1}\in F_2$ with $x_n\neq x'_n$ contradicting the fact that
$\frac{\omega_n}{\omega_1}\notin F_2$.

Therefore,
$$
|D_2|\le|I_2|\cdots|I_{n-1}|
$$
and
$$
\varpi\le\frac{|B|}{H_n}.
$$
If $H_n>p^{\tau}$, we are done. Otherwise
\begin{equation}\label{eq:H_{k+1}}
H_{k+1}\cdots H_n\le p^{(n-k)\tau} < p^{\tau}.
\end{equation}

Define
\begin{equation}\label{eq:B_2}
B_2=\left\{x_1+x_2\frac{\omega_2}{\omega_1}+\cdots+x_k\frac{\omega_k}{\omega_1} : x_i\in
I_i,\,1\le i\le k\right\}.
\end{equation}

Hence $B_2\subset F_2$ and by~\eqref{eq:H_{k+1}}
$$
|B_2|>\frac{|B|}{H_{k+1}\cdots H_n}>p^{\frac25-\frac{\tau}{2}n}>p^{\frac{n}{3}}.
$$

(We assume $\tau<\frac{2}{15}$.)

Clearly, if $(x_2,\ldots,x_n)\in D_2$, then
$z=x_{k+1}\frac{\omega_{k+1}}{\omega_1}+\cdots+x_n\frac{\omega_n}{\omega_1}\in F_2$. Assume
$\chi|_{F_2}$ non-principal. Then by completing the sum over $y$ and recalling the classical
estimates for Gaussian sums in finite fields~\cite[Theorem 5.11]{LN} and~\eqref{eq:B_2}, we have
$$
\left|\sum_{y\in B_2}\chi(y+z)e_p(\Tr(a\omega_1(y+z)))\right|\le(\log p)^{\frac{n}{2}}\max_{\psi} %
\left|\sum_{x\in F_2}\psi(x)\chi(x)\right|\le(\log p)^{\frac{n}{2}}|F_2|^{\frac12}\le p^{-\frac{n}{15}}|B_2|,
$$
where $\psi$ runs over all additive characters. Therefore, clearly
$$
\varpi\le H_{k+1}\cdots H_n p^{-\frac{n}{15}}|B_2|=p^{-\frac{n}{15}}|B|
$$
providing the required estimate.

If $\chi|_{F_2}$ is principal, then obviously
$$
\varpi = \left|\sum_{x_1\in I_1}\sum_{(x_2,\ldots,x_n)\in D_2} %
e_p(\Tr(a(x_1\omega_1+\cdots+x_n\omega_n)))\right| \le H_1 |D_2| = \left|F_2\cap\frac{1}{\omega_1}B\right|
$$
and
$$
\left|\sum_{x\in B}\chi(x)e_p(\Tr(ax))\right| = |\omega_1F_2\cap B| + O_n(p^{-\tau}|B|).
$$
This completes the proof of Theorem~\ref{thm:CG}.
\end{proof}

\section{Proof of Theorem~\ref{thm:KG}}

Similar to the proof of Theorem~\ref{thm:CG}, by breaking up $B$ in smaller boxes, we may assume
$$
H_1\asymp\cdots\asymp H_n\asymp p^{\frac14+\eps}.
$$

Then, using the arguments in the proof of Theorem~\ref{thm:CG}, we can prove Theorem~\ref{thm:KG}
along the lines of Konyagin~\cite{Kon}.

\section*{Acknowledgements}
The author thanks Professor Igor Shparlinski for his helpful comments on an earlier version of the paper. The author also thanks Professor Sergei Konyagin for sending his preprint.

Part of this work was done while the author was visiting the Morningside Center of Mathematics, whose hospitality is gratefully acknowledged. The author was supported by the National Natural Science Foundation of China (Grant No. 10671056).

\providecommand{\bysame}{\leavevmode\hbox to3em{\hrulefill}\thinspace}

\end{document}